\documentclass{scrartcl}
\usepackage{amsmath, amsthm, amssymb}
\newtheorem{thm}{Theorem}[section]
\newtheorem{prop}[thm]{Proposition}
\newtheorem{lem}[thm]{Lemma}

\newtheorem{rmk}[thm]{Remark}
\theoremstyle{definition}
\newtheorem{defn}[thm]{Definition}
\newtheorem*{notn}{Notation}
\theoremstyle{remark}

\DeclareMathOperator{\Hom}{Hom}
\DeclareMathOperator{\Stab}{Stab}
\newcommand{\qqq}{\mathbb{Q}}
\newcommand{\ccc}{\mathbb{C}}

\newcommand{\points}{p_1, \ldots, p_n}
\newcommand{\loops}{c_1, \ldots, c_n}
\newcommand{\ttuple}{(t_1, \ldots, t_n)}
\newcommand{\cjk}{\mathbb{C}_{(jk)}}
\newcommand{\vlg}{V^l_g\ttuple}
\newcommand{\vmg}{V^m_g\ttuple}
\newcommand{\lljk}{L^l_{jk}}
\newcommand{\llij}{L^l_{ij}}

\newcommand{\lqij}{L^q_{ij}}

\newcommand{\hstar}{H^*(S_g\ttuple; \qqq)}
\newcommand{\sgt}{S_g\ttuple}
\newcommand{\clljk}{c_1(L^l_{jk})}
\newcommand{\cllij}{c_1(L^l_{ij})}
\newcommand{\cllik}{c_1(L^l_{ik})}

\newcommand{\clqij}{c_1(L^q_{ij})}

\newcommand{\cqij}{c^q_{ij}}
\newcommand{\cqjk}{c^q_{jk}}
\newcommand{\cqik}{c^q_{ik}}
\newcommand{\sigpunc}{\Sigma \setminus \{ \points \}}
\newcommand{\aas}{a_1, \ldots, a_g}
\newcommand{\bbs}{b_1, \ldots, b_g}
\newcommand{\sljk}{s^l_{jk}}
\newcommand{\slik}{s^l_{ik}}
\newcommand{\slij}{s^l_{ij}}
\newcommand{\sljjk}{s^l_{j,k}}
\newcommand{\sliij}{s^l_{i,j}}
\newcommand{\sliik}{s^l_{i,k}}
\newcommand{\sljji}{s^l_{j,i}}
\newcommand{\slkki}{s^l_{k,i}}
\newcommand{\sljkx}{\sljk(x)}
\newcommand{\slijx}{\slij(x)}
\newcommand{\slikx}{\slik(x)}
\newcommand{\ccj}{\ccc_{(j)}}
\newcommand{\llj}{L^l_j}
\newcommand{\lli}{L^l_i}
\newcommand{\llk}{L^l_k}
\newcommand{\cllj}{c_1(\llj)}
\newcommand{\clli}{c_1(\lli)}
\newcommand{\cllk}{c_1(\llk)}
\newcommand{\rcm}{\rho(c_m)}
\newcommand{\rcl}{\rho(c_l)}
\newcommand{\rcq}{\rho(c_q)}
\newcommand{\rhx}{\rho(x)}
\newcommand{\sgq}{\sigma_q}
\newcommand{\sgl}{\sigma_l}
\newcommand{\sqq}{s^q}
\newcommand{\sll}{s^l}
\title{Moduli spaces of rank 3 parabolic bundles over a many-punctured surface}
\author{Elisheva Adina Gamse\thanks{Partially supported by NSERC grant PDF-488168} \\Department of Mathematics, University of Toronto } 
\date{18 March 2019}
\begin{document}
\maketitle

\begin{abstract} Let $M$ be the moduli space of rank 3 parabolic vector bundles over a Riemann surface with several punctures. By the Mehta-Seshadri correspondence, this is the space of rank 3 unitary representations of the fundamental group of the punctured surface with specified conjugacy classes of the images of each boundary component, up to conjugation by elements of the unitary group. For each puncture we consider the torus bundle on $M$ consisting of those representations where the image of the corresponding boundary component is a fixed element of the torus. We associate line bundles to these torus bundles via one-dimensional torus representations, and consider the subring of the cohomology ring of $M$ generated by their first Chern classes. By finding explicit sections of these line bundles with no common zeros we give a geometric proof that particular products of these Chern classes vanish. We prove that the ring generated by these Chern classes vanishes below the dimension of the moduli space, in a generalisation of a conjecture of Newstead. 
\end{abstract}
\section{Introduction}
Parabolic vector bundles were introduced by Mehta and Seshadri in the 1970s (see \cite{sesh} and \cite{mspara}) to generalise the Narasimhan-Seshadri correspondence (\cite{nsstableunitary}) between stable vector bundles over a compact Riemann surface and irreducible unitary representations of its fundamental group to the non-compact case. 

Recall that the {\em slope} of a vector bundle is its degree divided by its rank. A vector bundle $E$ is said to be {\em stable} if for all subbundles $D < E$, we have $slope(D) < slope(E)$; this stability condition was introduced by Mumford in \cite{mumford} as restricting attention to stable bundles makes the moduli spaces well-behaved. A {\em parabolic} vector bundle over $\Sigma$ is a stable vector bundle over $\Sigma$ together with a choice of flag in the fibre above each marked point (or each cusp, in the non-compact case considered by Mehta-Seshadri). 

Building on the results of Atiyah and Bott for moduli spaces of vector bundles \cite{ab}, Nitsure \cite{nitsure} computed the Betti numbers of the cohomology of the moduli space of parabolic bundles, and showed that when the rank is coprime to the degree of the bundles, the cohomology is torsion-free. Earl and Kirwan \cite{ek2} provided a complete set of relations between the generators of $H^*$ listed by Atiyah and Bott. 

In \cite{jw}, Weitsman focused on the case of rank 2 parabolic bundles over Riemann surfaces $\Sigma$ of genus $g \geq 2$. He considered tautological bundles $L \to \Sigma$, and provided a topological proof that 
\begin{equation} \label{rank2}
c_1(L)^{2g}=0
\end{equation}
as conjectured by Newstead in \cite{newstead}. In \cite{aj} we gave a generalisation of \eqref{rank2} to moduli spaces of parabolic vector bundles of rank $n$ over Riemann surfaces with only one marked point. The purpose of this paper is to prove a generalisation of \eqref{rank2} to moduli spaces of rank 3 parabolic vector bundles over Riemann surfaces with finitely many marked points. 

Specifically, let $g \geq 2$ and let $\Sigma$ be a Riemann surface of genus $g$. Let $p_1, \ldots, p_n$ be points on $\Sigma$. We fix a presentation 
\begin{equation*}
\pi_1(\Sigma \setminus \{ \points \}) = \left\langle a_1, \ldots, a_g, b_1, \ldots, b_g, \loops \mid \prod_{i=1}^g [a_i, b_i] = \prod_{j=1}^n c_j \right\rangle
\end{equation*}
of the fundamental group of the $n$-punctured surface $\sigpunc$. 
Let $G$ be $SU(3)$ and choose the maximal torus $T$ consisting of diagonal matrices. Let $(t_1, \ldots, t_n)$ be an $n$-tuple of elements of $T$ that is {\em generic} (in a sense that will be made precise in Definition \ref{generic}). Then the moduli space of rank 3 parabolic vector bundles over $\Sigma$ is given by 
\begin{equation*}
S_g(t_1, \ldots, t_n) := \{ \rho \in \Hom(\pi_1(\Sigma \setminus \{\points\}), G) \mid \rho(c_i) \sim t_i \forall i\} / G,
\end{equation*}
where $\sim$ denotes conjugacy in $G$, and the quotient is by the conjugation action. For each $1 \leq l \leq n$, define the torus bundle
\begin{equation*}
\vlg := \{ \rho \in \Hom(\pi_1(\Sigma \setminus \{ \points \}), G) \mid \rho(c_l)=t_l \text{ and }\rho(c_i) \sim t_i \forall i\}.
\end{equation*}
Let $\cjk$ denote the torus representation 
\begin{align*}
T \times \ccc & \to \ccc \\
\left(\left(\begin{array}{ccc} e^{i \theta_1} && \\& e^{i\theta_2}&\\&& e^{i\theta_3} \end{array}\right), z\right) &\mapsto e^{i(\theta_j - \theta_k)}z.
\end{align*} 
For each $1 \leq l \leq n$, and each $1 \leq j,k \leq 3$, let $L^l_{jk}$ be the line bundle associated to $\vlg$ via the representation $\cjk$, and consider its first Chern class $c_1(L^l_{jk})$. (Observe that when $j=k$ we get a trivial bundle.) In this paper we prove the following.

\begin{thm} \label{main}
The product 
\begin{equation*}
\prod_{\substack{1 \leq l \leq n \\ 1 \leq j,k \leq 3 \\ j \neq k}} c_1(\lljk)^{d^l_{jk}}
\end{equation*}
vanishes in $\hstar$ whenever  
\begin{equation*}
\sum_{\substack{1 \leq l \leq n \\ 1 \leq j,k \leq 3 \\ j \neq k}} d^l_{jk} \geq 6g + 4n -5.
\end{equation*}
\end{thm}
\begin{rmk}
The dimension of $\sgt$ is $(2g+n)\dim G - \dim G  - n \dim T - \dim G = 16g + 6n -16$. Thus Theorem \ref{main} tells us that the ring generated by the $\clljk$ vanishes below the top cohomology of the moduli space whenever $n < 2g-3$. 
\end{rmk}

As in \cite{aj}, the method of proof of Theorem \ref{main} is to exhibit certain collections of sections of the $\lljk$ with no common zeros, enabling us to conclude that the corresponding products of Chern classes vanish in $\hstar$. Having identified such collections, we will give a combinatorial argument which uses the relations between the line bundles to show that any monomial in the $\clljk$ of sufficiently high degree is equivalent in $\hstar$ to some combination of those particular monomials already shown to vanish, and hence also vanishes. 

We will use the following throughout the paper. 
\begin{notn} If $m$ is a positive integer we will denote by $[m]$ the set $\{1, \ldots, m\}$. 
\end{notn}
\begin{defn} \label{generic} The $n$-tuple $\ttuple$ of elements $t_j \in T$ is said to be {\em generic} if 
\begin{enumerate} \item $\Stab(t_j) = T \quad \forall j \in [n]$, and \item If $\lambda_i$ is an eigenvalue of $t_i$ for each $1 \leq i \leq n$, then $\lambda_1 \cdots \lambda_n \neq 1$. \end{enumerate}
\end{defn}

\section{Identifying particular monomials which vanish}

Recall the following useful lemma:

\begin{lem} \label{topology} Suppose $\mathcal{L}_1, \ldots, \mathcal{L}_k \to M$ are complex line bundles over a manifold $M$, and $s_i: M \to \mathcal{L}_i$ are sections such that $\cap_{i=1}^k s^{-1}(0) = \varnothing$. Then $c_1(\mathcal{L}_1) \cdots c_1(\mathcal{L}_k) = 0$. 
\end{lem}
\begin{proof} Consider the vector bundle $E:= \mathcal{L}_1 \oplus \cdots \oplus \mathcal{L}_k$. The section $\sigma: M \to E$ given by $(s_1, \ldots, s_k)$ is nowhere zero, so $e(E)=0$. Thus $0=c_k(E) = c_1(\mathcal{L}_1) \cdots c_1(\mathcal{L}_k)$. 
\end{proof}

Our approach to proving Theorem \ref{main} is to identify particular sections of the $\lljk$ which have no common zeros, thus deducing the vanishing of the corresponding monomials in the $\clljk$ by Lemma \ref{topology}.  Recall that a $T$-equivariant map $f: V \to \ccc_{(\alpha)}$ from a torus bundle $V$ to a 
$T$-representation $\ccc_{(\alpha)}$ induces a section $\tilde{f}: V/T \to (V \times \ccc_{(\alpha)})/T$ of the associated line bundle (where the latter quotient is by the diagonal $T$-action).  

\begin{defn} Let $1 \leq l \leq n$. For each generator $x \in \{\aas, \bbs, c_1, \ldots, \hat{c_l}, \ldots, c_n \} $ of $\pi_1(\sigpunc)$ other than $c_l$, let $\sljkx$ be the section of $\lljk$ induced by the $T$-equivariant map 
\begin{align*}
\vlg &\to \cjk \\
\rho &\mapsto (\rho(x))_{jk},
\end{align*}
where the subscript denotes the $(j,k)^{th}$ matrix entry.
\end{defn}
The following lemma is a special case that will illustrate our use of these sections. 
\begin{lem} \label{basic} The monomial $\left(c_1(L^1_{12})c_1(L^1_{13})\right)^{2g+n-1}$ vanishes in $\hstar$. 
\end{lem}
\begin{proof}Consider the collection of sections 
\begin{equation} \label{inlemsecs} \left\{s^1_{12}(x), s^1_{13}(x) \mid x \in \{ \aas, \bbs, c_2, \ldots, c_n \}\right\}.
\end{equation} These sections all vanish on those $\rho \in V^1_g\ttuple $ where the images $\rho(x) \in SU(3)$ are of the form $\left( \begin{array}{ccc}  e^{i\phi} & 0 & 0 \\ 0 & z & w  \\ 0 & -\bar{w} & \bar{z} \end{array} \right)$ for every $x \in \{ \aas, \bbs, c_2, \ldots, c_n \}$. (Recall that $\rho(c_1) = t_1$ for $\rho \in V^1_g\ttuple$.) Hence for such $\rho$, 
\begin{equation} \label{inlem}
\prod_{j=1}^n \rho(c_j) = \prod_{i=1}^g [\rho(a_i), \rho(b_i)] = \left( \begin{array}{ccc}  1 & 0 & 0 \\ 0 & \zeta & \eta  \\ 0 & -\bar{\eta} & \bar{\zeta} \end{array} \right)
\end{equation} for some $\zeta, \eta \in \ccc$ with $\vert \zeta \vert ^2 + \vert \eta \vert ^2 = 1$. Notice also that if $\rho(c_j) = \left( \begin{array}{ccc}  e^{i\phi} & 0 & 0 \\ 0 & z & w  \\ 0 & -\bar{w} & \bar{z} \end{array} \right)$, then $e^{i\phi}$ is an eigenvalue $\lambda_j$ of $t_j$, since $\rho(c_j) \sim t_j$. Thus \eqref{inlem} also gives us a relation $\lambda_1 \cdots \lambda_n=1$ between eigenvalues $\lambda_j$ of the $t_j$. Since we picked $\ttuple$ to be generic (as in Definition \ref{generic}), there are no $\rho$ in $V^1_g\ttuple$ where such a relation holds. Thus the locus on which these $4g+2n-2$ sections \eqref{inlemsecs} all vanish is empty, and $(c_1(L^1_{12})c_1(L^1_{13}))^{2g+n-1}=0$ as claimed.\end{proof}

Observe that the Chern classes appearing in Lemma \ref{basic} were Chern classes of line bundles all associated to the same torus bundle $V^1_g\ttuple$. In order to work with products of Chern classes of line bundles associated to different torus bundles, it will be convenient to consider sections of the following line bundles (whose first Chern classes we will later express in terms of the $\clljk$). 

For $1 \leq j \leq 3$, let $\ccc_{(j)}$ denote the torus representation 
\begin{align*}
T \times \ccc & \to \ccc \\
\left(\left(\begin{array}{ccc} e^{i \theta_1} && \\& e^{i\theta_2}&\\&& e^{i\theta_3} \end{array}\right), z\right) &\mapsto e^{i\theta_j}z,
\end{align*} 
and let $L^l_j := (\vlg \times \ccj)/T$ be the line bundle associated to $\vlg$ via this representation. 
\begin{defn} For each $1 \leq l \leq n$, $1 \leq j, k \leq 3$, and $i \in [n] \setminus \{l\}$, we define a section $s^l{j,k}(c_i): \sgt \to L^l_j$ as follows. Given a class in $\sgt$, pick a representative $\rho \in \vlg$, observing that $\rho(c_l)=t_l$ and $\rho(c_i) \sim t_i$. That is, there is some $A \in SU(3)$ such that $A \rho(c_i) A^{-1} = t_i$, and the columns of $A^{-1}=A^*$ are eigenvectors of $\rho(c_i)$. Since our choice of $\ttuple$ was generic, the eigenvalues of $\rho(c_i)$ are distinct, and $A$ is determined up to left multiplication by elements of the torus.  We declare the image of $[\rho]$ under $\sljjk(c_i)$ to be $[(A \cdot \rho, A_{jk})] \in \llj$. This is well-defined, since $[(A\cdot \rho, A_{jk})] = [((tA)\cdot \rho, (tA)_{jk})]$ in $(\vlg \times \ccj)/T$. 

Observe that $A_{jk}=0 \iff \overline{A_{kj}} = (A^{-1})_{kj} = 0$, and so this section vanishes on $\rho$ if and only if the $k^{th}$ entry in any eigenvector of $\rho(c_i)$ with eigenvalue $(t_i)_{jj}$ is zero.
\end{defn}
We now collect some results concerning collections of sections which have the same zero locus in $\sgt$. 

\begin{lem} \label{transpose}Let $l$ and $q$ be distinct elements of $[n]$ and let $(i,j,k)$ be a permutation of $(1,2,3)$. An element $[\rho]$ in $\sgt$ is in the vanishing locus of $\sliij(c_q)$ and of $\sliik(c_q)$ if and only if it is in the vanishing locus of $\sljji(c_q)$ and of $\slkki(c_q)$. 
\end{lem}
\begin{proof}
Suppose $tA\rho(c_q)A^{-1}t^{-1}=t_q$. Observe that since $tA \in SU(3)$, we know $(tA)_{ij}=(tA)_{ik}=0$ if and only if $(tA)_{ji}=(tA)_{ki}=0$. 
\end{proof}
\begin{lem} \label{switch}
Let $l \neq m \in [n]$, let $A$ be a subset of $\{\aas, \bbs, c_j \mid j \in [n] \setminus \{l,m\}\}$, and let $(i,j,k)$ and $(u,v,w)$ be permutations of $(1,2,3)$. Then the collection of sections
\begin{equation} \label{coll1}
\{ \slij(x), \slik(x) \mid x \in A \} \cup \{s^l_{u,j}(c_m), s^l_{u,k}(c_m)\}
\end{equation} has the same vanishing locus in $\sgt$ as the collection
\begin{equation} \label{coll2}
\{s^m_{uv}(x), s^m_{uw}(x) \mid x \in A \} \cup \{s^m_{i,v}(c_l), s^m_{i,w}(c_l)\}
\end{equation} does. 
\end{lem}
\begin{proof}
The vanishing locus of the collection \eqref{coll1} consists of those elements $[\rho] \in \sgt$ that have representatives $\rho \in \vlg$ such that $(\rho(x))_{ij}=(\rho(x))_{ik}=0$ for all $x \in A$, and that also satisfy the condition that the $j$ and $k$ coordinates of the eigenvector of $\rho(c_m)$ with eigenvalue $(t_m)_{uu}$ are both zero. Since the $\rho(x)$ lie in $SU(3)$, this means that $\rho(c_m)$ and each $\rhx$ for $x \in A$ are block diagonal (with zeros in the non-diagonal entries of the $i^{th}$ row and column), and further, that $(\rcm)_{ii}$ is the eigenvalue $(t_m)_{uu}$. Given $\rho \in \vlg$, let $h_\rho \in SU(3)$ be a block diagonal matrix with the same shape as $\rcm$ such that $h_\rho \cdot \rcm$ is diagonal, and let $\sigma_\rho$ be the permutation matrix such that $\sigma_\rho \cdot h_\rho \cdot \rcm = t_m$. Then $(\sigma_\rho h_\rho) \cdot \rho$ is a representative of $[\rho]$ that lies in $\vmg$. Observe that $\sigma_\rho$ must send $i^{th}$ basis vector to the $u^{th}$ basis vector, since the $u^{th}$ eigenvalue of $t_m$ is in the $(i,i)$ position in $\rcm$. Furthermore, for each $x\in A$, the matrix $(\sigma_\rho h_\rho) \cdot \rhx$ must be block diagonal with zeros in the non-diagonal entries of the $u^{th}$ row and column. Finally, since $\rho(c_l)=t_l$, we know that $(\sigma_\rho h_\rho) \cdot \rcl$ has $(u,u)$ entry $(t_l)_{ii}$, and zeros elsewhere in the $u^{th}$ row and column. Thus, classes $[\rho] \in \sgt$ that are in the vanishing locus of \eqref{coll1} are also in the vanishing locus of \eqref{coll2}; by symmetry, the two vanishing loci are equal. 
\end{proof}

\begin{lem} \label{nozeros} Let $l \in [n]$ and let $X$ and $Y$ be disjoint sets such that $X\sqcup Y = \{c_i \mid i \in [n] \setminus \{ l \}\}$, let $(i,j,k)$ be a permutation of $(1,2,3)$, and let $f: [n] \to [3]$ be a function. Then the collection of sections 
\begin{equation}
\left\{\slij(x), \slik(x) \mid x \in \{\aas, \bbs\} \cup X\right\} \cup \left\{ s^l_{f(q),j}(c_q), s^l_{f(q),k}(c_k) \mid q \in Y\right\}
\end{equation} has no common zeros.
\end{lem}
\begin{proof}
If $Y$ is empty then the proof is as in Lemma \ref{basic}. Suppose otherwise. If $[\rho]$ is in the vanishing locus of each of the sections in $\{s^l_{f(q),j}(c_q), s^l_{f(q),k}(c_q) \mid q \in Y \}$, then for each $q \in Y$, any eigenvector of $\rho(c_q)$ with eigenvalue $(t_q)_{f(q)f(q)}$ has zeros in the $j^{th}$ and $k^{th}$ coordinates. That is, the $i^{th}$ standard basis vector of $\ccc^3$ is an eigenvector of $\rcq$, and so $\rcq$ is block diagonal with blocks $\{i\}$ and $\{j,k\}$. If the sections $\left\{\slijx, \slikx \mid x \in \{\aas, \bbs\} \cup X\right\}$ also vanish on $[\rho]$, then each $\rhx$ is also blcok diagonal with the same structure. As in the proof of Lemma \ref{basic}, this would imply a relation between the eigenvalues of the $t_i$, contradicting our assumption that $\ttuple$ is generic; thus there are no $[\rho] \in \sgt$ on which these sections all vanish. 
\end{proof}

Given a torus bundle $V$ and line bundles $L_\phi, L_\psi$ associated to $V$ via torus representations $\phi, \psi$, observe that $L_\phi \otimes L_\psi \cong L_{\phi + \psi}$ and $L^*_\phi \cong L_{-\phi}$; hence $c_1(L_\phi) + c_1(L_\psi) = c_1(L_{\phi + \psi}$ and $c_1(L_{-\phi}) = -c_1(L_\phi)$. In particular, the Chern classes we consider in this paper obey the following relations:

\begin{equation} \label{relns}
\begin{aligned}
c_1(\llij) + \clljk &= \cllik \\
\cllij &= -c_1(L^l_{ji}) \\
\cllij &= \clli - \cllj \\
\cllij + \cllik &= 3\clli \\
\clli + \cllj + \cllk &= 0
\end{aligned}
\end{equation}
The following proposition uses the earlier results in this section together with the relations \eqref{relns} to exhibit a class of monomials in the $\cllij$ which vanish. 
\begin{prop} \label{zetavanishes} Let $r_1, \ldots, r_n$ be non-negative integers with $r_1 + \ldots + r_n = 2g + n - 1$. For each $1 \leq q \leq n$, let $\sigma_q$ be a permutation of $(1,2,3)$. Then the monomial 
\begin{equation} \label{zetaprop}
\zeta = \prod_{q=1}^n \left(c_1(L^q_{\sgq(1)\sgq(2)})c_1(L^q_{\sgq(1)\sgq(3)})\right)^{r_q}
\end{equation}
vanishes in $\hstar$. 
\end{prop}
\begin{proof}
First observe that for a permutation $\sigma$ of $(1,2,3)$, 
\begin{align*}
c_1(L^q_{\sigma(1)\sigma(2)})c_1(L^q_{\sigma(1)\sigma(3)}) &= \left(c_1(L^q_{\sigma(1)})-c_1(L^q_{\sigma(2)})\right)\left(c_1(L^q_{\sigma(1)})-c_1(L^q_{\sigma(3)})\right) \\
&=c_1(L^q_{\sigma(1)})^2 - c_1(L^q_{\sigma(1)})\left(c_1(L^q_{\sigma(2)})+c_1(L^q_{\sigma(3)})\right) + c_1(L^q_{\sigma(2)})c_1(L^q_{\sigma(3)}) \\
&= 2c_1(L^q_{\sigma(1)})^2 + c_1(L^q_{\sigma(2)})c_1(L^q_{\sigma(3)}).
\end{align*}
Let $Q \subseteq [n]$ be the subset $\{q \in [n] \mid r_q \geq 1\}$, and rewrite $\zeta$ as 
\begin{equation*}
\zeta = \prod_{q\in Q} \left(c_1(L^q_{\sgq(1)\sgq(2)})c_1(L^q_{\sgq(1)\sgq(3)})\right)^{r_q - 1}\left(2c_1(L^q_{\sgq(1)})^2 + c_1(L^q_{\sgq(2)})c_1(L^q_{\sgq(3)})\right).
\end{equation*}
For each term in this polynomial we wish to find a collection of sections of the line bundles appearing in that term with no common zeros.

Pick $l \in Q$. By Lemma \ref{transpose}, for any $q \in Q \setminus \{l\}$ the section $\left(s^q_{\sgl(1),\sgq(2)}(c_l), s^q_{\sgl(1),\sgq(3)}(c_l)\right)$ of the line bundle $L^q_{\sgq(2)} \oplus L^q_{\sgq(3)}$ has the same zero locus as the section $\left(s^q_{\sgl(2),\sgq(1)}(c_l), s^q_{\sgl(3),\sgq(1)}(c_l)\right)$ of the bundle $(L^q_{\sgq(1)})^{\oplus 2}$. Thus it suffices to find a non-vanishing section of 
\begin{equation*}
\bigoplus_{q \in Q}\left(\left(L^q_{\sgq(1)\sgq(2)}\oplus L^q_{\sgq(1)\sgq(3)}\right)^{\oplus r_q - 1}\oplus L^q_{\sgq(2)} \oplus L^q_{\sgq(3)}\right).
\end{equation*}
Reordering the bundles if necessary, we may assume $Q = \left[\vert Q \vert \right]$ and choose $l = \vert Q \vert$. Take $X = \{x_1, \ldots, x_{2g+n-l}\}$, where 
\begin{equation*}
x_i = \begin{cases} a_i & 1 \leq i \leq g \\ b_{i-g} & g+1 \leq i \leq 2g \end{cases},
\end{equation*}
and $\{x_{2g+1}, \ldots, x_{2g+n-l}\} =  \{c_j \mid j \notin Q\}$. 
Consider the collection of sections 
\begin{equation} \label{first}
\begin{split}
\bigcup_{q=1}^{l-1} \{ \sqq_{\sgl(1),\sgq(2)}(c_l), \sqq_{\sgl(1),\sgq(3)}(c_l)\sqq_{\sgq(1)\sgq(2)}(x_i), \sqq_{\sgq(1)\sgq(3)}(x_i) \mid r_1 + \cdots + r_{q-1} < i \leq r_1 + \cdots + r_q - q \} \\\cup \{\sll_{\sgl(1)\sgl(2)}(x_i), \sll_{\sgl(1)\sgl(3)}(x_i) \mid r_1 + \cdots + r_{l-1} - l + 1 < i \leq r_1 + \cdots + r_l - l + 1 \}. 
\end{split}
\end{equation}
By Lemma \ref{switch}, this collection \eqref{first} has the same zero locus as the collection 
\begin{equation} \label{second}
\begin{split}
\bigcup_{q=1}^{l-1} \{ \sll_{\sgq(1),\sgl(2)}(c_q), \sll_{\sgq(1),\sgl(3)}(c_q)\sll_{\sgl(1)\sgl(2)}(x_i), \sll_{\sgl(1)\sgl(3)}(x_i) \mid r_1 + \cdots + r_{q-1} < i \leq r_1 + \cdots + r_q - q \} \\\cup \{\sll_{\sgl(1)\sgl(2)}(x_i), \sll_{\sgl(1)\sgl(3)}(x_i) \mid r_1 + \cdots + r_{l-1} - l + 1 < i \leq r_1 + \cdots + r_l - l + 1 \}. 
\end{split}
\end{equation}
This latter collection \eqref{second} may be more concisely expressed as 
\begin{equation} \label{concise}
\begin{split}
\left\{\sll_{\sgl(1)\sgl(2)}(x), \sll_{\sgl(1)\sgl(3)}(x) \mid x \in \{ \aas, \bbs, c_j \mid j \notin Q\} \right\} \\ \cup \bigcup_{q \in Q \setminus \{l\}}\{\sll_{\sgq(1),\sgl(2)}(c_q), \sll_{\sgq(1),\sgl(3)}(c_q)\}.
\end{split}
\end{equation}
By Lemma \ref{nozeros}, this collection \eqref{concise} has no common zeros, so we are done.
\end{proof}
%%%%%%%%%%%%%%%%%%%%%%%%%%%%%%%%%%%%%%%%%%%%%%%%%
\section{Combinatorial proof of the main theorem}
%%%%%%%%%%%%%%%%%%%%%%%%%%%%%%%%%%%%%%%%%%%%%%%%%
In this section we will show that a product 
\begin{equation*}
\prod_{\substack{1 \leq i,j \leq 3 \\ 1 \leq q \leq n}}\clqij^{d^q_{ij}}
\end{equation*}
is equivalent in $\hstar$ to a linear combination of monomials of the form \eqref{zetaprop} shown to vanish in Proposition \ref{zetavanishes} whenever $\sum_{q,i,j}d^q_{ij}$ is sufficiently large. To simplify notation, let $\cqij := \clqij$. Recall that for each $1 \leq q \leq n$ and distinct $i,j,k \in \{1,2,3\}$, we have $\cqij = -c^q_{ji}$ and $\cqij + \cqjk = \cqik$. We begin with a combinatorial lemma.
\begin{lem} \label{comb}
Fix $q \in [n]$ and let $R^q$ be the subring of $\hstar$ generated by the $\cqij$ (for $1 \leq i \neq j \leq 3$). Let $d$ be a positive integer and suppose $\zeta = (c^q_{12})^a (c^q_{23})^b (c^q_{31})^c$, where $a+b+c \geq 3d-1$. Then there exist polynomials $\phi_1, \phi_2, \phi_3 \in R^q$ such that 
\begin{equation} \label{desired}
\zeta = \phi_1 (c^q_{12}c^q_{13})^d + \phi_2 (c^q_{23}c^q_{21})^d + \phi_3 (c^q_{31}c^q_{32})^d.
\end{equation}
\end{lem}
\begin{proof} At least one of $a,b,c$ must be $\geq d$ (otherwise $a+b+c \leq 3d-3$); without loss of generality assume $a \geq d$. Then 
\begin{align*}
\zeta &= (c^q_{12})^d(-c^q_{31}-c^q_{23})^{a-d}(c^q_{23})^b(c^q_{31})^c \\
&= \sum_{r=0}^{a+b+c-d} \binom{a+b+c-d}{r} (c^q_{12})^d(-1)^{a-d}(c^q_{31})^{a-d-r+b}(c^q_{23})^{r+c}.
\end{align*}
Observe that for each $r$, we have $a-d-r+b \geq d$ or $r+c \geq d$, as otherwise $a+b+c \leq 3d-2$. Hence $\zeta$ is in the desired form \eqref{desired}.
\end{proof}
We are now ready to prove our main theorem.
\begin{thm} \label{main}
Let $d^q_{ij}$ be a non-negative integer for each $1\leq q \leq n$ and each pair $(i,j)$ of distinct elements in $\{1,2,3\}$. Then the product 
\begin{equation*}
\prod_{\substack{1 \leq q \leq n \\ 1 \leq i,j \leq 3 \\ i \neq j}} c_1(\lqij)^{d^q_{ij}}
\end{equation*}
vanishes in $\hstar$ whenever  
\begin{equation*}
\sum_{\substack{1 \leq q \leq n \\ 1 \leq i,j \leq 3 \\ i\neq j}} d^q_{ij} \geq 6g + 4n -5.
\end{equation*}
\end{thm}
\begin{proof} Let 
\begin{equation*}
\zeta = \prod_{\substack{1 \leq q \leq n \\ 1 \leq i,j \leq 3 \\ i \neq j}} c_1(\lqij)^{d^q_{ij}},
\end{equation*}
where each $d^q_{ij}$ is a nonnegative integer and $\sum_{\substack{1 \leq q \leq n \\ 1 \leq i,j \leq 3 \\ i\neq j}} d^q_{ij} \geq 6g + 4n -5$. Factorise $\zeta$ as $\zeta = \zeta_1 \cdots \zeta_n$, where for each fixed $q$, $\zeta_q$ is a product of monomials $\cqij$. For each $1 \leq q \leq n$, let $r_q$ be the largest integer such that $\sum_{1 \leq i,j \leq 3} d^q_{ij} \geq 3r_q -1$. By Lemma \ref{comb}, we may rewrite each $\zeta_q$ as 
\begin{equation*}
\zeta_q = \phi^q_1 (c^q_{12}c^q_{13})^{r_q} + \phi_2 (c^q_{23}c^q_{21})^{r_q} + \phi_3 (c^q_{31}c^q_{32})^{r_q};
\end{equation*} taking the product $\prod_{q=1}^n \zeta_q$ we get an expression for $\zeta$ in which each term is a monomial of the form 
\begin{equation} \label{zetaform}
\prod_{q=1}^n \phi^q_{\sigma(1)}\left(c^q_{\sgq(1)\sgq(2)}c^q_{\sgq(1)\sgq(3)}\right)^{r_q},
\end{equation} where each $\sigma_q$ is a permutation of $(1,2,3)$. 
Observe that 
\begin{equation*}
\sum_{\substack{1 \leq q \leq n \\ 1 \leq i,j \leq 3 \\ i \neq j}}d^q_{ij} \leq \sum_{q=1}^n
(3r_q+1) = 3(r_1+ \cdots + r_n) +n.\end{equation*}
If $r_1 + \cdots + r_n < 2g+n-1$, then 
\begin{equation*}
\sum_{\substack{1 \leq q \leq n \\ 1 \leq i,j \leq 3 \\ i \neq j}}d^q_{ij} \leq 3(2g+n-2) + n = 6g+4n-6,\end{equation*}
contradicting our assumption that $\sum_{\substack{1 \leq q \leq n \\ 1 \leq i,j \leq 3 \\ i\neq j}} d^q_{ij} \geq 6g + 4n -5$. So $r_1 + \cdots + r_n \geq 2g+n-1$. Hence by Proposition \ref{zetavanishes}, each term \eqref{zetaform} in $\zeta$ vanishes in $\hstar$. 
\end{proof}
\small
\bibliographystyle{plain}

\bibliography{/home/adina/Dropbox/Maths/Math/refs}

\end{document}